\documentclass[12pt]{amsart}

\usepackage[matrix,arrow,curve,frame]{xy}
\usepackage{amsmath,amsthm,amssymb,enumerate}
\usepackage{latexsym}
\usepackage{amscd}
\usepackage[colorlinks=true]{hyperref}
\usepackage{euscript}

\newtheorem{thm}[subsection]{Theorem}
\newtheorem{defn}[subsection]{Definition}

\newtheorem{claim}[subsection]{Claim}
\newtheorem{corr}[subsection]{Corollary}

\newtheorem{remark}{Remark}

\theoremstyle{definition}

\newcommand{\cat}{\mathcal}

\newcommand{\Q}{\mathbb Q}
\newcommand{\Z}{\mathbb Z}
\newcommand{\F}{\mathbb F}

\newcommand{\C}{\mathbb C}

\newcommand{\gG}{\mbox{\german g}}

\newcommand{\gC}{{\cat S}}

\DeclareMathOperator{\hocolim}{hocolim}

\DeclareMathOperator{\BK}{BK}
\DeclareMathOperator{\BH}{BH}
\DeclareMathOperator{\BN(T)}{BN(T)}
\DeclareMathOperator{\BW}{BW}
\DeclareMathOperator{\EK}{EK}
\DeclareMathOperator{\BT}{BT}

\DeclareMathOperator{\To}{T}
\DeclareMathOperator{\W}{W}
\DeclareMathOperator{\K}{K}
\DeclareMathOperator{\Hg}{H}

\DeclareMathOperator{\FFT}{N_{\mathbb{W}}(\overline{\F}_p)}
\DeclareMathOperator{\BFT}{BN(T)}
\DeclareMathOperator{\BFFT}{BN_{\mathbb{W}}(\overline{\F}_p)}

\DeclareMathOperator{\Hr}{H}

\DeclareMathOperator{\Gr}{G}
\DeclareMathOperator{\BG}{BG}

\DeclareMathOperator{\No}{N}

\newfont{\german}{eufm10}

 \DeclareMathOperator{\Tor}{Tor}

\newcommand\qu{/\kern-.7ex/}

\begin{document}
\pagestyle{plain}

\title
{On some applications of unstable Adams operations to the topology of Kac-Moody groups}
\author{Nitu Kitchloo}
\address{Department of Mathematics, Johns Hopkins University, Baltimore, USA}
\email{nitu@math.jhu.edu}
\thanks{Nitu Kitchloo is supported in part by NSF through grant DMS
  1307875.}

\date{\today}

{\abstract

\noindent
Localized at almost all primes, we describe the structure of differentials in several important spectral sequences that compute the cohomology of classifying spaces of topological Kac-Moody groups. In particular, we show that for all but a finite set of primes, these spectral sequences collapse and that there are no additive extension problems. We also describe some appealing consequences of these results. The main tool is the use of the naturality properties of unstable Adams operations on these classifying spaces.}
\maketitle

\tableofcontents

\section{Introduction}

\medskip
\noindent
The theory of Kac-Moody groups is well established at this point \cite{K1,K2, Ku}. The complex points of Kac-Moody groups form a natural extension of the class of semi-simple Lie groups, even though they need not be finite dimensional. Concepts like maximal torus, Weyl groups and root systems extend almost by definition to these groups and one may even describe the theory of highest weight representations of Kac-Moody groups along these lines. As topological groups, one may study Kac-Moody groups through algebraic invariants like cohomology \cite{Ki2}. From the standpoint of homotopy theory, the natural object related to a Kac-Moody group $\K$ is its classifying space $\BK$. The study of these classifying spaces was begun by the author in his doctoral thesis, and much work has been done in this area by several authors since. The results on the structure of the classifying spaces of Kac-Moogy groups show a striking similarity with those for Lie groups. For instance, the cohomology rings (with primary coefficients) in question are Noetherian and all the standard algebraic decomposition theorems known for classifying spaces of compact Lie groups extend to Kac-Moody groups \cite{BrK}. 

\medskip
\noindent
It was shown in \cite{Ki,Ki2}, that the space $\BK$ can be expressed as a homotopy colimit of classifying spaces of compact Lie groups denoted by $\BK_J$, where $J$ is an object in a suitable finite indexing poset. In particular, one gets a natural spectral sequence known as the Bousfield--Kan spectral sequence \cite{BK} that converges to the cohomology of $\BK$ and has an $E_2$-term given by higher inverse limits of the functor $\Hr^\ast(\BK_J)$. The purpose of this article is to show that under some very general assumptions, this spectral sequence collapses, and one can recover a lot of useful information about $\Hr^\ast(\BK)$ from this fact. 

\medskip
\noindent
The main tool used in this article is the action of the unstable Adams operations on the Bousfied--Kan spectral sequence. These operations were show to exist in \cite{F}. For the convenience of the reader, we shall describe the construction of these operations in Sections \ref{global} and \ref{local}. The main body of the article is organized as follows: In section \ref{back} we shall describe the basic background in the theory of Kac-Moody groups that is relevant in our context. This will allow us to state the main results. In Sections \ref{global} and \ref{local} we recall the construction of the unstable Adams operations, and in Section \ref{BKSS} we apply these operations to prove the main results. 

\medskip
\noindent
Before we begin, we would like to thank all the participants of ``Beyond $p$-compact groups", for stimulating our interest in these questions. Particular thanks to Natalia Castellana and Jesper Grodal for the invitation to participate. The author also wishes to thank an anonymous referee for helpful feedback. Finally, the author would like to thank Haynes Miller for this constant interest in this project and also for pointing out the need to be explicit about the construction given in Section \ref{local}.

\section{Background and Statement of results} \label{back}

\medskip
\noindent
In the theory of Kac-Moody groups, one begins with a finite integral matrix $A = (a_{i,j})_{i,j \in I}$ with the properties that $a_{i,i} = 2$ and $a_{i,j} \leq 0$ for $i\neq j$. Moreover, we demand that $a_{i,j} = 0$ if and only if $a_{j,i} = 0$. These conditions define a generalized Cartan Matrix. A generalized Cartan matrix is said to be symmetrizable if it becomes symmetric after multiplication with a suitable rational diagonal matrix. 

\noindent
Given a generalized Cartan matrix $A$, one may construct a complex Lie algebra $\gG (A)$ using the Harishchandra-Serre relations. The structure theory for the highest weight representations of $\gG(A)$ leads to a construction (in much the same way that Chevalley groups are constructed), of a topological group $\Gr(A)$ called the (minimal, split) Kac-Moody group over the complex numbers. The group $\Gr(A)$ supports a canonical anti-linear involution $\omega$, and one defines the unitary form $\K(A)$ as the fixed group $\Gr(A)^{\omega}$. We refer the reader to \cite{Ku} for details.

\medskip
\noindent
Given a subset $J \subseteq I$, one may define a parabolic subalgebra $\gG_J(A) \subseteq \gG(A)$. One may exponentiate these subalgebras to parabolic subgroups $\Gr_J(A) \subset \Gr(A)$. We then define the unitary Levi factors $\K_J(A)$ to be the groups $\K(A) \cap \Gr_J(A)$. Let $\Hg_J(A)$ denote the complexification of $\K_J(A) \subset \Gr_J(A)$. Hence the Lie algebra of $\Hg_J(A)$ corresponds to the semi simple factor in $\gG_J(A)$. Notice that $\K_{\varnothing} (A) = \To$ is a torus of rank $2|I| - rk(A)$, called the maximal torus of $\K(A)$. Indeed, it is the common maximal torus to all the groups $\K_J(A)$. The normalizer $\No(\To)$ of $\To$ in $\K(A)$, is an extension of a discrete group $W(A)$ by $\To$. The Weyl group $\W(A)$ has the structure of a crystallographic Coxeter group generated by reflections $r_i, i \in I$. For $J \subseteq I$, let $\W_J(A)$ denote the subgroup generated by the corresponding reflections $r_j, j \in J$. The group $\W_J(A)$ is a crystallographic Coxeter group in its own right that can be identified with the Weyl group of $\K_J(A)$. 

\medskip
\noindent
The inclusions $\K_J(A) \subset \Hg_J(A) \subset \Gr_J(A)$ are homotopy equivalences for all $J$. In this article we will study the classifying space of the topological group $\K(A)$. Some of the arithmetic arguments we use will require replacing $\K(A)$ with $\Gr(A)$. This does not pose any problem since $\Gr(A)$ and $\K(A)$ are homotopy equivalent. 

\medskip
\noindent
Given a generalized Cartan matrix $A = (a_{i,j})_{i,j \in I}$, define a category $\gC(A)$ to be the poset category (under inclusion) of subsets $J \subseteq I$ such that $\K_J(A)$ is a compact Lie group. This is equivalent to demanding that $\W_J(A)$ is a finite group. Notice that $\gC(A)$ contains all subsets of $I$ of cardinality less than two. In particular, $\gC(A)$ is nonempty and has an initial object given by the empty set. However, $\gC(A)$ need not have a terminal object unless $\K(A)$ itself is a compact Lie group. The category $\gC(A)$ is also known as the poset of spherical subsets \cite{D}. 

\medskip
\noindent
The topology on the group $\K(A)$ is the strong topology generated by the compact subgroups $\K_J(A)$ for $J \in \gC(A)$ \cite{Ki2,Ku}. More precisely, $\K(A)$ is the amalgamated product of the compact Lie groups $\K_J(A)$, in the category of topological groups. For a arbitrary subset $L \subseteq I$, the topology induced on the homogeneous space of the form $\K(A)/\K_L(A)$ makes it into a CW-complex, with only even cells, indexed by the set of cosets $\W(A)/\W_L(A)$. Furthermore, the projection maps $\K(A)/\K_L(A) \rightarrow \K(A)/\K_J(A)$ are cellular for $L \subseteq J$.

\begin{defn} Define the (finite-type) Topological Tits building $X(A)$ as the $K(A)$-space:
\[ X(A) =  {\hocolim}_{J \in \gC(A)} \;  \K(A)/\K_J(A).  \]
Similarly, define the (spherical-type) Davis complex, $D(A)$ to be the $\W(A)$-space \cite{D}:
\[ D(A) = {\hocolim}_{J \in \gC(A)} \; \W(A)/\W_J(A). \]
\end{defn}
\noindent 
Notice that by construction, $X(A)$ is a $\K(A)$-CW complex such that all the isotropy subgroups are compact Lie groups. It is well known \cite{Ki} that the space $X(A)$ is equivalent to the classifying space $\EK(A)$ for proper $\K(A)$-actions. Similarly, the Davis complex $D(A)$ is a model for the classifying space for proper $\W(A)$-actions, as can be seen by identifying it with the $\To$-fixed set of $X(A)$ (also see \cite{Ki}). In particular, the spaces $X(A)$ and $D(A)$ are contractible (see remark \ref{RBKX} below). 

\medskip
\noindent
Taking homotopy orbits of the $\K(A)$ and $\W(A)$-action on the spaces $X(A)$, and $D(A)$ respectively we get homotopy decompositions:
\[ \BK(A) = \hocolim_{J \in \gC(A)} \; \BK_J(A), \quad \quad \BW(A) = \hocolim_{J \in \gC(A)} \; \BW_J(A). \]

\noindent
Our main theorem is about the Bousfield--Kan spectral sequence induced by the homotopy decomposition for $\BK(A)$:
\begin{thm} \label{T1}
Let $q$ be a prime so that $\Hr^\ast(\BK_J(A);\Z)$ has no $q$-torsion for all $J \in \gC(A)$, then in the Bousfield--Kan spectral sequence:
\[ {\lim}^i \Hr^j(\BK_J(A); \Z_{(q)}) \Rightarrow \Hr^{i+j}(\BK(A); \Z_{(q)}), \]
the differential $d_k$ is trivial if $k$ is even. In addition, $d_{2r-1}$ is trivial unless $(q-1)$ divides $(r-1)$. In particular, the first possible non-trivial differential is $d_{2q-1}$. The above result also holds with coefficients in any $\Z_{(q)}$-module. 
\end{thm}
\noindent
Let $\K(A)$ be a Kac-Moody group with generalized Cartan matrix indexed on a set $I$ of cardinality $n$. Then the Bousfield--Kan spectral sequence only supports differentials $d_k$ for $k < n$ (for dimensional reasons), we will show:
\begin{corr} \label{C1}
Let $A$ be a generalized Cartan matrix indexed on a set $I$ of cardinality $n$. Assume that $\Hr^\ast(\BK_J(A);\Z)$ has no $q$-torsion for all $J \in \gC(A)$, and that $2q \geq n+1$. Then the Bousfield--Kan spectral sequence collapses with coefficients in the ring $\Z_{(q)}$. Furthermore, there are no additive extension problems. \footnote{Note: the term $E_2^{i,j}$ in the spectral sequence may be nontrivial for $i > 0$, and may contain $q$-torsion!} 
\end{corr}

\noindent
The inclusion map of the maximal torus $\BT \longrightarrow \BK(A)$ induces a map in cohomology $\Hr^\ast(\BK(A)) \longrightarrow \Hr^\ast(\BT)$. We have:
\begin{corr} \label{C2}
Assume that $\Hr^\ast(\BK_J(A);\Z)$ has no $q$-torsion for all $J \in \gC(A)$, and that the Bousfield--Kan spectral sequence collapses with coefficients in $\Z_{(q)}$. Then the restriction map $\Hr^\ast(\BK(A);\Z_{(q)}) \longrightarrow \Hr^\ast(\BT;\Z_{(q)})$ has image given by the Weyl invariants $\Hr^\ast(\BT,\Z_{(q)})^{\W(A)}$. Furthermore, the kernel of the restriction map is the ideal of nilpotent elements in $\Hr^\ast(\BK(A); \Z_{(q)})$. The above result also holds with coefficients in any $\Z_{(q)}$-module. 
\end{corr}

\begin{remark}
It is easy to show that the Weyl group $\W(A)$ has no elements of $q$-torsion for $q > n+1$ \cite{Ki3}. In particular, the same holds for the subgroups $\W_J(A)$ for all $J \in \gC(A)$. For such primes, it is well know that $\Hr^\ast(\BK_J(A),\Z)$ has no $q$-torsion, and so all the results stated above apply. 
\end{remark}

\noindent
In \cite{Ki2} and \cite{Ki3} we showed that the map $\BN(T) \longrightarrow \BK(A)$ induces a chomology isomorphism with coefficients in $\Z_{(q)}$ for all primes $q$ so that $\W(A)$ has no elements of $q$-torsion. We use this fact to show:
\begin{corr} \label{C3}
Let $q$ be a prime so that $W(A)$ has no elements of $q$-torsion. Consider the Bousfield--Kan spectral sequence with coefficients in $\Z_{(q)}$, converging to $\Hr^\ast(\BK(A), \Z_{(q)})$, and the Serre spectral sequence for the fibration:
\[ \BT \longrightarrow \BN(T) \longrightarrow \BW(A). \]
Then the collapse of any one of these spectral sequences implies the collapse of the other. Furthermore, one may identify the $E_2$-terms: ${\lim}^i \Hr^j(\BK_J(A), \Z_{(q)})$ with $\Hr^i(\W(A), \Hr^j(\BT, \Z_{(q)}))$. The above result also holds with coefficients in any $\Z_{(q)}$-module. 
\end{corr}

\begin{remark} \label{RBKX}
The reader may inquire if the Bousfield-Kan spectral sequence for the topological space $X(A)$ collapses. Indeed, we show in claim \ref{BKX} that this spectral sequence collapses at the $E_2$-term with any coefficients.
\end{remark}

\section{The unstable Adams operations} \label{global}

\medskip
\noindent
Given a compact connected Lie group $\K$, let $p$ be a prime which does not divide the order of the Weyl group $\W(\K)$. Recall \cite{JMO} that there is a unique map (up to homotopy): $\psi : \BK \longrightarrow \BK$, with the property that the induced map $\psi^\ast : \Hr^{2k}(\BK, \Q) \longrightarrow \Hr^{2k}(\BK, \Q)$ is given by multiplication by $p^k$. The map $\psi$ goes by the name of {\em Unstable Adams operation for the prime $p$}. Let us begin this section by reproving the following result:

\begin{thm} \cite{F} Let $\K(A)$ denote a Kac-Moody group corresponding to the Cartan matrix $A$. Let $p$ be a prime so that the Weyl group: $\W(A)$ contains no elements of order $p$. Then there exists a family of unstable Adams operations for the prime $p$: $\psi_J: \BK_J(A) \longrightarrow \BK_J(A)$ which are compatible (not just up to homotopy) with respect to the inclusions $J \leq L$ in $\gC(A)$. In particular, one has a global map $\psi : \BK(A) \longrightarrow \BK(A)$ compatible with the maps $\psi_J$ under the inclusion $\BK_J(A) \longrightarrow \BK(A)$. 
\end{thm}

\begin{remark}
We call $\psi$ the unstable Adams operation on $\BK(A)$ for the prime $p$. Note that we do not claim that $\psi$ is unique up to homotopy, or that it is diagonal with respect to homogeneous decomposition of $\Hr^\ast(\BK(A), \Q)$. 
\end{remark}

\medskip
\noindent
To construct the global unstable Adams operation $\psi$, we will invoke results of Tits \cite{T}, and proceed along the lines described by the work of Friedlander-Mislin \cite{FM}. Recall that Tits has constructed a functor $\Gr_R(A)$ from the category of commutative rings to (discrete) groups, that depends on the root datum defined by the generalized Cartan matrix $A$, and yields the (minimal, split) Kac-Moody group when the ring $R$ is a field. This construction also gives rise to a functor $\Hg_{J,R}(A)$ that realizes the group $\Hg_{J}(A)$ as the complex points (endowed with a suitable topology). 

\medskip
\noindent
Now let $p$ be a prime so that $\W(A)$ contains no element of order $p$. Let $\mathbb{W}(\overline{\F}_p)$ denote the ring of Witt vectors over the algebraic closure $\overline{\F}_p$ of $\F_p$. We fix once and for all, an embedding of $\mathbb{W}(\overline{\F}_p)$ in the complex numbers $\C$. In particular, one obtains a map:
\[ \Gr_{\mathbb{W}(\overline{\F}_p)}(A) \longrightarrow \Gr_{\C}(A) \rightarrow \Gr(A), \]
with the second map above being a continuous bijection of groups.

\medskip
\noindent
Recall that Tits \cite{T} has constructed a functorial group $\tilde{\W}_R(A)$ that lifts the Weyl group in $\Gr_R(A)$. We use this to define the group $\FFT \subset \Gr_{\mathbb{W}(\overline{\F}_p)}(A)$ to be the group generated by $\tilde{\W}_{\mathbb{W}(\overline{\F}_p)}(A)$, and the torsion subgroup of the maximal torus $\To_{\mathbb{W}(\overline{\F}_p)}\subset \Gr_{\mathbb{W}(\overline{\F}_p)}(A)$. Notice that the $\FFT$ maps naturally to $\No(\To)$ under the chosen embedding. In addition, the Frobenius map agrees with the degree $p$-map on $\To$, through this embedding.

\medskip
\noindent
Let us now define the space $\BN(T)_\psi$ via the pullback:
\[
\xymatrix{
\BN(T)_\psi     \ar[d] \ar[r]^{\phi_1  \quad \quad \quad} & \BFFT \hat{}_{\Q_q} \times \BFT \hat{}_{\Q_p} \ar[d] \\
\BFT_\Q    \ar[r] & \BFT \hat{}_\Q,}
\]
where $\BFT\hat{}_\Q$ denotes its Adelic completion (i.e. the rationalization of its $\Z$-completion), and the space $\BFT\hat{}_{\Q_p}$ denotes the rationalization of its $p$-completion. Similarly, the space $\BFFT \hat{}_{\Q_q}$ denotes the rationalization of the completion of $\FFT$, at all primes $q \neq p$. The bottom horizontal map $\BFT_\Q  \longrightarrow \BFT\hat{}_\Q$ is the natural map from the rationalization of $\BFT$ to the Adelic completion of $\BFT$. 

\medskip
\noindent
It is easy to see that both vertical maps in the above diagram are homotopy equivalences. The merit of replacing $\BN(T)_\Q$ by the equivalent space $\BN(T)_\psi$ is that we obtain a natural automorphism $\psi : \BN(T)_\psi \longrightarrow \BN(T)_\psi$ that interpolates the self map $\psi_p$ of $\BFT$ induced by the degree $p$ map on $\To$, and the self map of $\BFFT$ induced by the Frobenius $\mbox{F}_p$ on $\overline{\F}_p$ on the respective vertices of the diagram. 

\medskip
\noindent
Now define the space $\BH_R(A)$ as the homotopy colimit:
\[ \BH_R(A) := \hocolim_{J \in \gC(A)} \BH_{J,R}(A). \]

\medskip
\noindent
One also has a commutative diagram:
\[
\xymatrix{
\BH_{\mathbb{W}(\overline{\F}_p)}(A)\hat{}_q \times \BFT\hat{}_p    \ar[d]^{\phi_2} \ar[r]& \BG(A)\hat{}_q \times \BG(A) \hat{}_p \ar[d] \\
\BH_{\mathbb{W}(\overline{\F}_p)}(A)\hat{}_{\Q_q} \times \BFT\hat{}_{\Q_p}  \ar[r]& \BG(A)\hat{}_{\Q_q} \times \BG(A) \hat{}_{\Q_p},}
\]
where the vertical maps are given by rationalization, and the horizontal maps are induced by the inclusion $\mathbb{W}(\overline{\F}_p) \subset \C$. The homotopy decoposition of $\BK(A)$, and classical results from \cite{FM} imply that the horizontal maps are homotopy equivalences (see remark \ref{levi} below). In particular, the space $\BH_{\mathbb{W}(\overline{\F}_p)}(A)\hat{}_q \times \BFT\hat{}_p$ is a model for $\BG(A)\hat{}$ that admits a self map $\psi$ induced by the Frobenius $\mbox{F}_p$ on $\Gr_{\mathbb{W}(\overline{\F}_p)}(A)$, and the self map of $\BFT$ induced by the degree $p$ map on $\To$. 

\medskip
\noindent
Finally, define the pullback:
\[
\xymatrix{
\BG(A)_{\psi} \ar[d] \ar[r] & \BH_{\mathbb{W}(\overline{\F}_p)}(A)\hat{}_q \times \BFT\hat{}_p \ar[d]^{\phi_2} \\
\BN(T)_\psi \ar[r]^{\phi_1 \quad \quad \quad \quad} & \BH_{\mathbb{W}(\overline{\F}_p)}(A)\hat{}_{\Q_q} \times \BFT\hat{}_{\Q_p}.}
\]
By construction, the above pullback is equivalent to the arithmetic fracture square for $\BG(A)$. In other words, $\BG(A)_\psi$ is homotopy equivalent to $\BG(A)$, which we know to be homotopy equivalent to $\BK(A)$. Furthermore, by construction, $\BG(A)_\psi$ supports a self map $\psi$ interpolating the corresponding self maps on the vertices of the above diagram. 

\medskip
\noindent
Notice that the above constructions are natural with respect to the root data defined by $A$. In particular, the map $\psi$ restricts to a self map of $\BH_J(A)$. This restriction is indeed equivalent to the unstable Adams operation $\psi_J$ on $\BK_J(A)$. 

\begin{remark} \label{levi}
It is more natural to use the space $\BG_{\mathbb{W}(\overline{\F}_p)}(A)$ instead of $\BH_{\mathbb{W}(\overline{\F}_p)}(A)$ in the above argument. However, showing that $\BG_{\mathbb{W}(\overline{\F}_p)}(A) \hat{}_q$ is homotopy equivalent to $\BG(A)\hat{}_q$ would require proving the acyclicity of certain parabilic unipotent subgroups in $\Gr_{\mathbb{W}(\overline{\F}_p)}(A)$ away from $p$. This is indeed true but requires an involved proof; see \cite{F} for details. We circumvent these issues here by working with the spaces $\BH_{\mathbb{W}(\overline{\F}_p)}(A)$ instead. 
\end{remark}

\bigskip
\noindent
\noindent\section{The $q$-local unstable Adams Operation} \label{local}

\bigskip
\noindent
In this section, we prove a local version of the unstable Adams operation that was constructed in the previous section. The argument is essentially identical to the one used earlier. Indeed, the construction of the local unstable Adams operation given below is implicit in the construction of the global one from the previous section:

\begin{thm} \label{main} Let $\K(A)$ denote a Kac-Moody group corresponding to the Cartan matrix $A$. Let $p$ and $q$ be distinct primes. Then there exists a family of unstable Adams operations for the prime $p$: $\psi_J: \BK_J(A)_{(q)} \longrightarrow \BK_J(A)_{(q)}$ which are compatible (not just up to homotopy) with respect to the inclusions $J \leq L \leq I$. In particular, one has a map $\psi := \psi_I : \BK(A)_{(q)} \longrightarrow \BK(A)_{(q)}$ compatible with the maps $\psi_J$ under the inclusion $\BK_J(A)_{(q)} \longrightarrow \BK(A)_{(q)}$. 
\end{thm}

\bigskip
\noindent
Let $p$ be a prime. Working with the fixed embedding of $\mathbb{W}(\overline{\F}_p)$ in the complex numbers $\C$ chosen in the previous section, one obtains maps:
\[ \Hg_{J, \mathbb{W}(\overline{\F}_p)}(A) \longrightarrow \Hg_{J, \C}(A) \longrightarrow \Hg_J(A), \]
where $\Hg_J(A)$ denotes the group $\Hg_{J,\C}(A)$ with the analytic topology, so that the second map above is a continuous bijection of groups.

\medskip
\noindent
As pointed out earlier, Tits \cite{T} has constructed groups $\tilde{\W}_{J.R}(A)$ that lifts the Weyl group in $\Hg_{J, R}(A)$. Moreover, this construction is functorial in the ring $R$ and the root datum $A$. We use this to define the group $\FFT \subset \Hg_{J, \mathbb{W}(\overline{\F}_p)}(A)$ to be the group generated by $\tilde{\W}_{J, \mathbb{W}(\overline{\F}_p)}(A)$, and the torsion subgroup of the (common) maximal torus $\To_{\mathbb{W}(\overline{\F}_p)}\subset \Hg_{J, \mathbb{W}(\overline{\F}_p)}(A)$. Notice that the $\FFT$ maps naturally to the normalizer $\No_J(\To)$ of $\To$ in $\Hg_J(A)$ under the chosen embedding. In addition, the Frobenius map agrees with the degree $p$-map on $\To$, through this embedding.

\medskip
\noindent
Let us now define the space $\BN(T)_{J,\psi}$ via the pullback:
\[
\xymatrix{
\BN(T)_{J,\psi}    \ar[d] \ar[r]^{\quad \phi_1  \quad \quad} & \BFFT \hat{}_{\Q_q}  \ar[d] \\
\BFT_\Q    \ar[r] & \BFT \hat{}_{\Q_q}.}
\]

\medskip
\noindent
Notice that both vertical maps in the above diagram are homotopy equivalences. The merit of replacing $\BFT_\Q$ by the equivalent space $\BN(T)_{J,\psi}$ is that we obtain a natural automorphism $\psi_J : \BN(T)_{J,\psi} \longrightarrow \BN(T)_{J,\psi}$ that interpolates the self map of $\BFT$ induced by the degree $p$ map on $\To$, and the self map of $\BFFT$ induced by the Frobenius $\mbox{F}_p$ on $\overline{\F}_p$ on the respective vertices of the diagram. 

\medskip
\noindent
Now, define the pullback:
\[
\xymatrix{
\BH_J(A)_{\psi} \ar[d] \ar[r] & \BH_{J,\mathbb{W}(\overline{\F}_p)}(A)\hat{}_q  \ar[d] \\
\BN(T)_{J,\psi} \ar[r]^{\quad \phi_1 \quad \quad \quad} & \BH_{J,\mathbb{W}(\overline{\F}_p)}(A)\hat{}_{\Q_q} .}
\]
The results of Friedlander-Mislin \cite{FM} can be used to show that the above pullback is equivalent to the arithmetic fracture square for $\BH_J(A)_{(q)}$ (see \cite{F}). In other words, $\BH_J(A)_\psi$ is homotopy equivalent to $\BH_J(A)_{(q)}$, which we know to be homotopy equivalent to $\BK_J(A)_{(q)}$. Furthermore, by construction, $\BH_J(A)_\psi$ supports a self map $\psi_J$ interpolating the corresponding self maps on the vertices of the above diagram. 

\medskip
\noindent
Again invoking Tits \cite{T}, we see that the above constructions are natural with respect to the root data defined by $A$. In other words, $\psi_J$ extends to the map $\psi_L$ under the inclusion $\BH_J(A)_{(q)} \rightarrow \BH_L(A)_{(q)}$ induced by $J \leq L \leq I$. Furthermore, all maps $\psi_J$ are restrictions of the global map $\psi := \psi_I$ acting on $\BH_I(A)_{(q)} = \BG(A)_{(q)}$, which is equivalent to $\BK(A)_{(q)}$. This proves theorem \ref{main}

\bigskip
\begin{remark}
A $q$-complete version of theorem \ref{main} was proven by J. Foley (see Theorems A, B in \cite{F}). The $q$-local version given above requires a little more work in glueing in the rational information along the fracture square. 
\end{remark}

\medskip
\noindent
\section{The Bousfield--Kan spectral sequence} \label{BKSS}

\medskip
\noindent
By the naturality of the construction of $\psi$ on $\BK(A)$, we get an induced action $\psi^\ast$ on the Bousfield--Kan spectral sequence with coefficients in the ring $\Z_{(q)}$ for some prime $q$. 

\medskip
\noindent
Now the $E_1$-term of the Bousfield--Kan spectral sequence has an explicit description:
\[ E_1^{i,j} = \bigoplus_{J_1 < \cdots < J_i} \Hr^j(\BK_{J_1}(A), \Z_{(q)}), \]
where we have used the standard simplicial resolution of $\gC(A)$ to calculate the derived functors of ${\lim}^0$. Therefore, if we assume that $\Hr^\ast(\BK_J(A), \Z_{(q)})$ has no $q$-torsion for all $J \in \gC(A)$, then the above group injects into its rationalization, and we observe that the action of $\psi^\ast$ on $E_s^{i,2j}$ is given by multiplication by $p^j$. Notice that if $R$ is any $\Z_{(q)}$-module, the lack of torsion implies that $\Hr^\ast(\BK_J(A), R) = \Hr^\ast(\BK_J(A), \Z_{(q)}) \otimes R$. In particular, the action of $\psi^\ast$ has the same description with coefficients in $R$.

\bigskip
\begin{remark}
It is an interesting question to ask if the unstable Adams operations for a prime $p$ are semi-simple in cohomology, with coefficients in a field. In other words, one would like to know if $\psi^\ast : \Hr^\ast(\BK(A), \F) \longrightarrow \Hr^\ast(\BK(A), \F)$ can be diagonalized into its eigenspaces. This question seems non-trivial even in characteristic zero. 
\end{remark}

\medskip
\noindent
{\bf The proof of theorem \ref{T1}}

\medskip
\noindent
In the absence of $q$-torsion, the groups $\Hr^\ast(\BK_J(A), \Z_{(q)})$ are evenly graded. Hence the differential $d_k$ is trivial for degree reasons if $k$ is even. Now let $\psi$ denote the $q$-local unstable Adams operation from theorem \ref{main}. Since $\psi^\ast$ commutes with the differentials, we have the equalities for $x \in E_{2r-1}^{i,2j}$:
\[ p^{j-r+1} d_{2r-1}(x) = \psi^\ast d_{2r-1}(x) = d_{2r-1} \psi^\ast(x) = d_{2r-1} p^j(x) = p^j d_{2r-1}(x). \]
In particular, for any prime $p$ so that $\W(A)$ has no elements of $p$-torsion, we have:
\[ p^{j-r+1}(p^{r-1}-1) d_{2r-1} (x) = 0. \]
Let us pick a prime $p$ that generates the cyclic group $(\Z/q)^\times$ and such that there are no elements of $p$-torsion in $\W(A)$. Then the above equality implies that $(q-1)$ must divide $(r-1)$. This proves theorem \ref{T1}. 

\begin{remark}
The above argument shows that the image of a non-zero differential consists of torsion elements with bounded exponent, and it even yields an upper bound on this exponent. 
\end{remark}

\medskip
\noindent
{\bf The proof of corollary \ref{C1}}

\medskip
\noindent
We now move to the proof of corollary \ref{C1}. The only statement that requires proof is the claim that there are no additive extension problems in the Bousfield--Kan spectral sequence. Let $q$ be a prime so that $2q \geq n+1$. In particular, the Bousfield--Kan spectral sequence with coefficients in any $\Z_{(q)}$-module collapses. In addition, assume that $\Hr^\ast(\BK_J(A),\Z_{(q)})$ has no $q$-torsion. Hence the the $E_1$-term of the Bousfield--Kan spectral sequence described above is a cochain complex of finitely generated free $\Z_{(q)}$-modules. The universal coefficient theorem applies to give us a short exact sequence:
\[ ({\lim}^i \Hr^\ast(\BK_J(A), \Z_{(q)})) \otimes \F_q \longrightarrow {\lim}^i \Hr^\ast(\BK_J(A), \F_q) \longrightarrow \Tor({\lim}^{i+1} \Hr^\ast(\BK_J(A), \Z_{(q)}), \F_q).\]
Summing over all indices $i$, we get a short exact sequence:
\[ 0 \rightarrow \bigoplus_{i+j=k} E_{\infty}^{i,j}(\Z_{(q)}) \otimes \F_q \longrightarrow \Hr^k(\BK(A),\F_q) \longrightarrow \bigoplus_{i+j=k+1} \Tor(E_{\infty}^{i,j}(\Z_{(q)}), \F_q) \rightarrow 0, \]
where $E_{\infty}^{i,j}(\Z_{(q)})$ denotes the $E_{\infty} = E_{2}$-term in the Bousfield--Kan spectral sequence with coefficients in $\Z_{(q)}$. We compare the above short exact sequence to the sequence:
\[ 0 \rightarrow \Hr^k(\BK(A), \Z_{(q)}) \otimes \F_q \longrightarrow \Hr^k(\BK(A), \F_q) \longrightarrow \Tor(\Hr^{k+1}(\BK(A), \Z_{(q)}), \F_q) \rightarrow 0. \]
Notice that any nontrivial extension in $E_{\infty}$ in total degree $k$, would strictly increase the dimension of the terms $E_{\infty}(\Z_{(q)}) \otimes \F_q$ in comparison with $\Hr^k(\BK(A), \Z_{(q)}) \otimes \F_q$. Similarly, it would also increase the dimension of $ \Tor(E_{\infty}(\Z_{(q)}), \F_q)$. This would contradict the fact that both short exact sequences have identical middle terms. We conclude that there can be no additive extensions, proving corollary \ref{C1}. 

\medskip
\noindent
{\bf The proof of corollary \ref{C2}}

\medskip
\noindent
The next corollary is \ref{C2}. Assume that $\Hr^\ast(\BK_J(A), R)$ has no $q$-torsion, and that the Bousfield--Kan spectral sequence with coefficients in a $\Z_{(q)}$-module $R$ collapses. Consider the restriction map $\Hr^\ast(\BK(A), R) \longrightarrow \Hr^\ast(\BT, R)$. It is clear that this map factors through the surjective edge homomorphism: $\Hr^\ast(\BK(A), R) \longrightarrow {\lim}^0 \Hr^\ast(\BK_J(A), R)$. Since our assumptions ensure that the restriction map $\Hr^\ast(\BK_J(A),R) \longrightarrow \Hr^\ast(\BT, R)^{\W_J(A)}$ is an isomorphism \cite{Fe}, we have the sequence of equalities:
\[  {\lim}^0 \Hr^\ast(\BK_J(A), R) =  \bigcap_{J \in \gC(A)} \Hr^\ast(\BT, R)^{\W_J(A)} = \bigcap_{i \in I} \Hr^\ast(\BT, R)^{\langle r_i \rangle} = \Hr^\ast(\BT,R)^{\W(A)}. \]

\noindent
Now notice that the kernel of the restriction map is supported in the columns $E_2^{i,j}$ with $i >0$. Since the Bousfield--Kan spectral sequence has a finite number of columns, its multiplicative structure implies that any element supported on a non-zero column must be nilpotent. In fact, any $n$-fold product of such elements is trivial, where $n$ is the size of the generalized Cartan matrix. This proves \ref{C2}}. 

\medskip
\noindent
{\bf The proof of corollary \ref{C3}}

\medskip
\noindent
Finally, we come to corollary \ref{C3}. Recall that the homotopy orbits of the $\W(A)$-action on the Davis complex $D(A)$ furnished us with a homotopy decomposition of $\BW(A)$ in terms of $\BW_J(A)$. One has a Bousfield--Kan spectral sequence for this decomposition, with coefficients in any $\W(A)$-module. More precisely, given a $\W(A)$-module $M$, we may construct a bundle of parametrized spectra over $\BW(A)$ with fibers being the Eilenberg-MacLane spectrum with coefficients in $M$. The space of global sections of this parametrized spectrum is a co-simplicial space dual to the canonical simplicial structure induced by the homotopy decomposition of $\BW(A)$. The homotopy groups of the space of sections is precisely the cohomology groups of $\W(A)$ with coefficients in $M$. Invoking the spectral sequence that converges to the homotopy groups of a co-simplicial space gives rise to the Bousfield--Kan spectral sequence:
\[ E_2^{i,j} = {\lim}^i \Hr^j(\W_J(A), M) \Rightarrow \Hr^{i+j}(\W(A), M). \]
Now assume that $M$ is a $\Z_{(q)}$-module, where $q$ is a prime so that there are no elements of $q$-torsion in $\W(A)$. In particular, $q$ is prime to the order of $\W_J(A)$ for all $J \in \gC(A)$. It follows that the above spectral sequence collapses to the invariants to give:
\[ \Hr^k(\W(A), M) = {\lim}^k \, M^{\W_J(A)}. \]
Let us apply this observation to the Serre spectral sequence, with coefficients in a $\Z_{(q)}$-module $R$, for the fibration:
\[ \BT \longrightarrow \BN(T) \longrightarrow \BW(A). \]
We know that the $E_2$-term is given by $\Hr^i(\W(A), \Hr^j(\BT, R))$. Assuming there are no elements in $\W(A)$ of $q$-torsion, we have:
\[ E_2^{i,j} = {\lim}^i \Hr^j(\BT, R)^{\W_J(A)}. \]
Under the assumptions on the prime $q$, $\Hr^j(\BT, R)^{\W_J(A)}$ is well known to be naturally isomorphic to $\Hr^j(\BK_J(A), R)$. It follows that the right hand side of the $E_2$-term above is isomorphic to the $E_2^{i,j}$-term of the Bousfield--Kan spectral sequence computing $\Hr^*(\BK(A), R)$. 

\smallskip
\noindent
Now recall from \cite{Ki2, Ki3} that for a prime $q$ that does not occur in the torsion of $\W(A)$, the localization of $\BN(T)$ agrees with that of $\BK(A)$. Therefore we notice that both spectral sequences converge to the same groups. Hence the collapse of one must imply the collapse of the other.

\medskip
\noindent
{\bf The Bousfield--Kan spectral sequence for the topological Tits building}

\medskip
\noindent
In this article, we have studied the Bousfield--Kan spectral sequence induced by the homotopy decomposition for $\BK(A)$. Recall that this homotopy decomposition was obtained by taking homotopy orbits with respect to the $\K(A)$-action on the topological Tits building $X(A)$. One may as well ask for the structure of the Bousfield--Kan spectral sequence for $X(A)$:
\[ E_2^{i,j} = {\lim}^i \Hr^j(\K(A)/\K_J(A), R) \Rightarrow \Hr^{i+j}(X(A), R). \]
For this spectral sequence, we have:
\begin{claim} \label{BKX}
Let $R$ denote any coefficients, then the terms $E_2^{i,j}$ in the Bousfield--Kan spectral sequence converging to $\Hr^\ast(X(A), R)$ are all zero for $i+j >0$. 
\end{claim}
\begin{proof}
Recall that we may describe the spaces $\K(A)/K_J(A)$ as CW complexes with only even cells (see the section with background). Furthermore, the projection maps given by $\K(A)/\K_J(A) \longrightarrow \K(A)/\K_L(A)$ are cellular for an inclusion $J < L$ in $\gC(A)$. It follows that the homotopy colimit $X(A)$ has a canonical CW structure induced from the cellular diagram. One checks that the $E_1$-term in the Bousfield--Kan spectral sequence can be identified with the cellular cochain complex for $X(A)$, and therefore the $E_2$-term is the singular cohomology of $X(A)$, with coefficients in $R$. This cohomology is trivial since $X(A)$ is contractible. 
\end{proof}

\pagestyle{empty}
\bibliographystyle{amsplain}
\providecommand{\bysame}{\leavevmode\hbox
to3em{\hrulefill}\thinspace}

\end{document}